\documentclass[preprint,12pt]{elsarticle}%

\usepackage{amssymb}
\usepackage{latexsym}
\usepackage{amsfonts}
\usepackage{amsmath}
\usepackage{amscd}
\usepackage{mathrsfs}
\usepackage{amsbsy}
\usepackage{amsthm}
\usepackage{graphicx}
\usepackage{mathtools}
\usepackage{hyperref}
\usepackage{color}
\providecommand{\U}[1]{\protect\rule{.1in}{.1in}}

\newcommand{\N}{\mathbb{N}}

\newcommand{\R}{\mathbb{R}}

\newcommand{\X}{\mathbb{X}}

\newcommand{\cF}{\mathcal{F}}

\newcommand{\cN}{\mathcal{N}}

\newcommand{\oY}{\overline{Y}}
\newcommand{\uY}{\underline{Y}}

\newcommand{\be}{\begin{equation}}
\newcommand{\ee}{\end{equation}}

\newcommand{\st}{\,:\,}

\providecommand{\U}[1]{\protect\rule{.1in}{.1in}}
\providecommand{\U}[1]{\protect\rule{.1in}{.1in}}
\newcommand{\nl}{\vskip 5pt\noindent}

\newtheorem{theorem}{Theorem}
\theoremstyle{plain}

\newtheorem{definition}{Definition}
\newtheorem{example}{Example}

\newtheorem{lemma}{Lemma}

\newtheorem{proposition}{Proposition}
\newtheorem{remark}{Remark}

\numberwithin{equation}{section}

\begin{document}

\begin{frontmatter}

\title{Bilinear Fractal Interpolation and Box Dimension}
\author{Michael F. Barnsley\fnref{anu}}
\fntext[anu]{Mathematical Sciences Institute, The Australian National University, Canberra, ACT, Australia, michael.barnsley@anu.edu.au, mbarnsley@aol.com}
\author{Peter R. Massopust\fnref{tum,hmgu}}
\fntext[tum]{Centre of Mathematics, Research Unit M6, Technische Universit\"at M\"unchen, Boltzmannstr. 3, 85747 Garching, Germany, massopust@ma.tum.de}
\fntext[hmgu]{Helmholtz Zentrum M\"unchen,Ingolst\"adter Landstra{\ss}e 1, 8764 Neuherberg, Germany}

\begin{abstract}
In the context of general iterated function systems (IFSs), we introduce
bilinear fractal interpolants as the fixed points of certain
Read-Bajraktarevi\'{c} operators. By exhibiting a generalized ``taxi-cab"
metric, we show that the graph of a bilinear fractal interpolant is the
attractor of an underlying contractive bilinear IFS. We present an explicit
formula for the box-counting dimension of the graph of a bilinear fractal
interpolant in the case of equally spaced data points. 
\end{abstract}


\end{frontmatter}

\section{Introduction}

Bilinear filtering or bilinear interpolation is used in computer graphics to
compute intermediate values for a two-dimensional regular grid. One of the
main objectives is the smoothening of textures when they are enlarged or
reduced in size. In mathematical terms, the interpolation technique is based
on finding a function $f(x,y)$ of the form $f(x,y) = a + bx + cy + d xy$,
where $a,b,c,d\in\mathbb{R}$, that passes through prescribed data points.

As textures reveal, in general, a non-smooth or even fractal characteristic,
a description in terms of fractal geometric methods seems reasonable. To this
end, the classical bilinear approximation method is replaced by a bilinear
fractal interpolation procedure. The latter allows for additional parameters,
such as the box dimension, that are related to the regularity and appearance
of an underlying texture pattern.

We introduce a class of fractal interpolants that are based on bilinear
functions of the above form. We do this by considering a more general class of
iterated function systems (IFSs) and by using a more general definition of
attractor of an IFS. These more comprehensive concepts are primarily based on
topological considerations. In this context, we extend and correct some known
results from \cite{massopust} concerning fractal interpolation functions that
are fixed points of so-called Read-Bajraktarevi\'{c} operators. Theorem
\ref{ifsthm} relates the fixed point in Theorem \ref{operatorthm} to the
attractor of an IFS and generalizes known results to the case where the IFS is
not contractive.

As a special example of the preceding theory we introduce bilinear fractal interpolants
and show that their graphs are the attractors of an underlying contractive bilinear IFS. Such bilinear IFSs have been investigated in \cite{BV} in connection with
fractal homeomorphisms and address structures underlying an IFS. Finally, we
present an explicit formula for the box dimension of the graph of a bilinear fractal
interpolant in the case where the data points are equally spaced.

\section{\label{general}General iterated function systems}

The terminology here for iterated function system, attractor, and contractive
iterated function system is from \cite{BVW}. Throughout this paper,
$(\mathbb{X},d)$ denotes a complete metric space with metric $d=d_{\mathbb{X}%
}$.

\begin{definition}
Let $N\in\mathbb{N} := \{1, 2, 3, \ldots\}$. If $f_{n}:\mathbb{X}\rightarrow\mathbb{X}$,
$n=1,2,\dots,N,$ are continuous mappings, then $\mathcal{F}=\left(
\mathbb{X};f_{1},f_{2},...,f_{N}\right)  $ is called an \textbf{iterated
function system} (IFS).
\end{definition}

By slight abuse of terminology we use the same symbol $\mathcal{F}$ for the
IFS, the set of functions in the IFS, and for the following mappings. We
define $\mathcal{F}:2^{\mathbb{X}}\rightarrow 2^{\mathbb{X}}$ by
\[
\mathcal{F}(B) := \bigcup_{f\in\mathcal{F}}f(B)
\]
for all $B\in2^{\mathbb{X}},$ the set of subsets of $\mathbb{X}$. Let
$\mathbb{H=H(X)}$ be the set of nonempty compact subsets of $\mathbb{X}$.
Since $\mathcal{F}\left(  \mathbb{H}\right)  \subset\mathbb{H}$, we can also
treat $\mathcal{F}$ as a mapping $\mathcal{F}:\mathbb{H} \rightarrow \mathbb{H}$. When
$U\subset\mathbb{X}$ is nonempty, we may write $\mathbb{H}(U)=\mathbb{H(X)}%
\cap2^{U}$. We denote by $\left\vert \mathcal{F}\right\vert $ the number of
distinct mappings in $\mathcal{F}$.

Let $d_{\mathbb{H}}$ denote the Hausdorff metric on $\mathbb{H}$, defined in
terms of $d_{\mathbb{X}}$. A convenient definition (see for example
\cite[p.66]{edgar}) is%
\[
d_{\mathbb{H}}(B,C) := \inf\{r>0:B\subset C+r,C\subset B+r\},
\]
for all $B,C\in\mathbb{H}$. For $S\subset\mathbb{X}$ and $r>0,$ $S+r$ denotes
the set $\{y\in\mathbb{X}:\exists x\in S$ $\;\,\text{so that}\;d_{\mathbb{X}}(x,y)<r\}$.

We say that a metric space $\mathbb{X}$ is \textit{locally compact} to mean
that if $C\subset\mathbb{X}$ is compact and $r$ is a positive real number then
$\overline{C+r}$ is compact. Here, $\overline{S}$ denotes the closure of a set $S$.
%
(For an equivalent definition of local compactness, see for instance \cite[3.3]{Engel}.)

The following information is foundational.

\begin{theorem}
\label{ctythm}
\begin{itemize}
\item[(i)] The metric space $(\mathbb{H},d_{\mathbb{H}})$ is complete.

\item[(ii)] If $(\mathbb{X},d_{\mathbb{X}})$ is compact then $(\mathbb{H}%
,d_{\mathbb{H}})$ is compact.

\item[(iii)] If $(\mathbb{X},d_{\mathbb{X}})$ is locally compact then $(\mathbb{H}%
,d_{\mathbb{H}})$ is locally compact.

\item[(iv)] If $\mathbb{X}$ is locally compact, or if each $f\in\mathcal{F}$ is
uniformly continuous, then $\mathcal{F}:\mathbb{H\rightarrow H}$ is continuous.

\item[(v)] If $f:\mathbb{X\rightarrow}\mathbb{X}$ is a contraction mapping for each
$f\in\mathcal{F}$, then $\mathcal{F}:\mathbb{H\rightarrow H}$ is a contraction mapping.
\end{itemize}
\end{theorem}

\begin{proof}
(i) This is well-known. A short proof can be found in \cite[p.67, Theorem
2.4.4]{edgar}.

(ii) This is well-known; see for example \cite{henrikson}. Here is a short
proof. Let $\varepsilon>0$ be given$.$ Since $\mathbb{X}$ is compact we can
find a finite set of points $\mathbb{X}_{\varepsilon}\subset\mathbb{X}$ such
that $\mathbb{X=\cup}_{x\in\mathbb{X}_{\varepsilon}}\mathcal{B}\left(
x,\varepsilon\right)  $ where $\mathcal{B}\left(  x,\varepsilon\right)
\subset\mathbb{X}$ denote the open ball with center at $x$ and radius
$\varepsilon$. Let $\mathbb{H}_{\varepsilon} := 2^{\mathbb{X}_{\varepsilon}}$, a
finite set of points in $\mathbb{H}$. It is readily verified that
$\mathbb{H=\cup}_{C\in\mathbb{H}_{\varepsilon}}\mathcal{B}\left(
C,\varepsilon\right)  $ where now $\mathcal{B}\left(  C,\varepsilon\right)
\subset\mathbb{H}$ denotes the open ball with center at $C\in\mathbb{H}$ and
radius $\mathbb{\varepsilon}$, measured using the Hausdorff metric. It follows
that $\mathbb{H}$ is totally bounded. It follows that $\mathbb{H}$ is compact.

(iii) Let $C\in\mathbb{H}$. Consider the set $\overline{C+r}$. It belongs to
$\mathbb{H}$ since $\mathbb{X}$ is locally compact. Let $\varepsilon>0$ be
given$.$ Since $\overline{C+r}$ is a compact subset of $\mathbb{X}$ we can
find a finite set of points $C_{\varepsilon}\subset\overline{C+r}$ such that
$\overline{C+r}\mathbb{\subset\cup}_{c\in C_{\varepsilon}}\mathcal{B}\left(
c,\varepsilon\right)  $. Let $\mathbb{C}_{\varepsilon} := 2^{C_{\varepsilon}}$, a
finite set of points in $\mathbb{H}$. It is readily verified that
$\overline{C+r}\mathbb{\subset\cup}_{c\in\mathbb{C}_{c}}\mathcal{B}\left(
C,\varepsilon\right)  $ where now $\mathcal{B}\left(  c,\varepsilon\right)
\subset\mathbb{H}$ denotes the open ball with center at $c\in\mathbb{H}$ and
radius $\mathbb{\varepsilon}$, measured using the Haudorff metric. It follows
that $\overline{C+r}$ is totally bounded. It follows that $\overline{C+r}$ is compact.

(iv) Let $B\in\mathbb{H}$. We show that $\mathcal{F}:\mathbb{H\rightarrow H}$
is continuous at $B$. We restrict attention to the action of $\mathcal{F}$
on$\ B+1$. If $\mathbb{X}$ is locally compact, it follows that $\overline
{B+1}$ is compact. It follows that each $f\in\mathcal{F}$ is uniformly
continuous on $\overline{B+1}.$ It follows that if $\mathbb{X}$ is locally
compact, or if each $f\in\mathcal{F}$ is uniformly continuous, we can find
$\delta_{\varepsilon}>0$ such that $d_{\mathbb{X}}(f(x),f(y))<\varepsilon$
whenever $d_{\mathbb{X}}(x,y)<\delta_{\varepsilon},$ for all $x,y\in$
$\overline{B+1}$, for all $f\in\mathcal{F}.$ Let $C\in\mathbb{H}$ with
$d_{\mathbb{H}}(B,C)<\delta_{\varepsilon}$ and let $f\in\mathcal{F}$. We can
suppose that $\delta_{\varepsilon}<1$.

Let $b^{\prime}\in f(B)$. Then there is $b\in B$ such that $f(b)=b^{\prime}$.
Since $d_{\mathbb{H}}(B,C)<\delta_{\varepsilon}$ there is $c\in C$ such that
$d(b,c)<\delta_{\varepsilon}$. Since $\delta_{\varepsilon}<1$ we have $c\in
B+1$. It follows that $d(f(b),f(c))<\varepsilon$. It follows that $f(B)\subset
f(C)+\varepsilon.$ By a similar argument $f(C)\subset f(B)+\varepsilon.$ Hence
$d_{\mathbb{H}}(f(B),f(C))<\varepsilon$.

(v) This is Hutchinson's theorem, \cite[p. 731]{hutchinson}, proved as
follows. Verify that if $\lambda\geq0$ is a uniform Lipschitz constant for all
$f\in\mathcal{F}$, namely $d_{\mathbb{X}}(f(x),f(y))\leq\lambda d_{\mathbb{X}%
}(x,y)$ for all $f\in\mathcal{F}$, for all $x,y\in\mathbb{X}$, then $\lambda$
is also a Lipschitz constant for $\mathcal{F}$, namely $d_{\mathbb{H}%
}(\mathcal{F}(B),\mathcal{F}(C))\leq\lambda d_{\mathbb{H}}(B,C)$ for all
$B,C\in\mathbb{H}$. If $f:\mathbb{X\rightarrow}\mathbb{X}$ is a contraction
mapping for each $f\in\mathcal{F}$, then we can choose $\lambda<1$. It follows
that $\mathcal{F}:\mathbb{H\rightarrow H}$ is a contraction mapping.
\end{proof}

For $B\subset\mathbb{X}$, let $\mathcal{F}^{k}(B)$ denote the $k$-fold
composition of $\mathcal{F}$, i.e., the union of $f_{i_{1}}\circ f_{i_{2}%
}\circ\cdots\circ f_{i_{k}}(B)$ over all finite words $i_{1}i_{2}\cdots i_{k}$
of length $k$ over the alphabet $\{1, \ldots, N\}$. Define $\mathcal{F}^{0}(B) := B.$

\begin{definition}
\label{attractdef}A nonempty compact set $A\subset\mathbb{X}$ is said to be an
\textbf{attractor} of the IFS $\mathcal{F}$ if
\begin{itemize}
\item[(i)] $\mathcal{F}(A)=A$ and

\item[(ii)] there exists an open set $U\subset\mathbb{X}$ such that $A\subset U$ and
$\lim_{k\rightarrow\infty}\mathcal{F}^{k}(B)=A,$ for all $B\in\mathbb{H(}U)$,
where the limit is with respect to the Hausdorff metric.
\end{itemize}
The largest open set $U$ such that $\mathrm{(ii)}$ is true is called the \textbf{basin of
attraction} (for the attractor $A$ of the IFS $\mathcal{F}$).
\end{definition}

Note that if $U_1$ and $U_2$ satisfy condition $\mathrm{(ii)}$ in Definition 2 for the same attractor $A$ then so does 
$U_1 \cup U_2$. We also remark that the invariance condition $\mathrm{(i)}$ is not needed; it follows from $\mathrm{(ii)}$ for $B := A$.

\begin{example}
An IFS $\cF = (\X; f_1, f_2, \ldots, f_N)$ is called \textbf{contractive} if each $f\in \cF$
is a contraction (with respect to the metric $d$), i.e., if there is a constant $s \in [0, 1)$
such that $d(f(x_1), f(x_2)) \leq s\,d(x_1, x_2)$, for all $x_1, x_2 \in \X$. By item $\mathrm{(v)}$ in Theorem 1, the mapping 
$\cF : \mathbb{H}(\X) \to \mathbb{H}(\X)$ is then also contractive on the complete metric space $(\mathbb{H}(\X); d_{\mathbb{H}})$ and thus possesses a unique attractor $A$. In this case, the basin of attraction for $A$ is $\X$.
\end{example}

We will use the following observation \cite[Proposition 3 (vii)]{lesniak},
\cite[Proposition 2.5.6]{edgar}.

\begin{lemma}
\label{intersectlemma}Let $\left\{  B_{k}\right\}  _{k=1}^{\infty}$ be a
sequence of nonempty compact sets such that $B_{k+1}\subset B_{k}$, for all
$k$. Then $\cap_{k\geq1}B_{k}=\lim_{k\rightarrow\infty}B_{k}$ where
convergence is with respect to the Haudorff metric.
\end{lemma}

%
\begin{theorem}
\label{attractorthm}Let $\mathcal{F}$ be an IFS with attractor $A$ and basin
of attraction $U.$ If $\mathcal{F}:\mathbb{H(}U)\mathbb{\rightarrow H(}U)$ is
continuous then%
\[
A=\bigcap\limits_{K\geq1}\overline{\bigcup_{k\geq K}\mathcal{F}^{k}(B)}%
\quad\text{ for all }B\subset U\text{ such that }\overline{B}\in
\mathbb{H(}U)\text{.}%
\]

\end{theorem}

The quantity on the right-hand side here is sometimes called the
\textit{topological upper limit }of the sequence $\left\{ \cF^{k}(B)\right\}
_{k=1}^{\infty}$.
\begin{proof}
Note that the proof can be carried out under the assumption that $B \in \mathbb{H}(U)$.
(It then follows from [12, Proposition 3(i)] that Theorem 2 is true for all $B \subset U$
such that $\overline{B} \in \mathbb{H}(U)$.) Under this assumption and with the fact that $\overline{\bigcup_k \cF^k (B)} =
\overline{\bigcup_k \overline{\cF^k(B)}}$ (see, for instance, \cite{Engel}) the statement follows from Theorem 3.82 in \cite{AB}.
\end{proof}

We will also need the following observation.

\begin{lemma}
\label{uniformlemma} Let $\mathbb{X}$ be locally compact. Let $\mathcal{F}%
=\left(  \mathbb{X};f_{1},f_{2},...,f_{N}\right)  $ be an IFS with attractor
$A$ and basin of attraction $U$. For any given $\varepsilon>0$ there is an
integer $L$ such that for each $x\in A+\varepsilon$ there is an integer $l\leq
L$ such that%
\[
d_{\mathbb{H}}(A,\mathcal{F}^{l}(\left\{  x\right\}  ))<\varepsilon.
\]

\end{lemma}

\begin{proof}
For each $x\in\overline{A+\varepsilon}$ there is an integer $l(x,\varepsilon)$
so that $d_{H}(A,\mathcal{F}^{l(x,\varepsilon)}(\left\{  x\right\}
))<\varepsilon/2$.

Since $\mathbb{X}$ is locally compact it follows that $\mathcal{F}%
^{l(x,\varepsilon)}:\mathbb{H\rightarrow H}$ is continuous. Since
$\mathcal{F}^{l(x,\varepsilon)}:\mathbb{H\rightarrow H}$ is continuous there
is an open neighborhood $N(\left\{  x\right\}  )$ (in $\mathbb{H}$) of
$\left\{  x\right\}  $ such that $d_{\mathbb{H}}(A,\mathcal{F}^{l(x,\varepsilon
)}(Y))<\varepsilon$ for all $Y\in N(\left\{  x\right\}  )$. It follows, in
particular, that there is an open neighborhood $N(x)$ (in $\mathbb{X}$) of $x$
such that $d_{\mathbb{H}}(A,\mathcal{F}^{l(x,\varepsilon)}(\left\{  y\right\}
))<\varepsilon$ for all $y\in N(x)$. Also since $\mathbb{X}$ is locally
compact, there is a finite set of points $\left\{  x_{1},x_{2},...,x_{q}%
\right\}  $ such that $\overline{A+\varepsilon}\subset\cup_{i=1}^{q}N(x_{i})$.
Choose $L := \max_{i}l(x_{i},\varepsilon)$.
\end{proof}

\section{Fractal interpolants as fixed points of operators}

Let $\{(X_{j},Y_{j}):j=0,1,...,N\}$ denote the cartesian coordinates of a
finite set of points in the Euclidean plane, with
\[
X_{0}<X_{1}<...<X_{N}\text{.}%
\]
Let $I$ denote the closed interval $[X_{0},X_{N}]$. For $n=1,2,...,N$, let
$l_{n}:I\rightarrow\lbrack X_{n-1},X_{n}]$ be a continuous bijection. Let
$L:I\rightarrow I$ be such that%
\[
L(x) :=l_{n}^{-1}(x)\text{ for }x\in [X_{n-1},X_{n})
\]
for $n=1,2,...,N$ (with the tacit understanding that for $n=N$ the interval is $[X_{N-1},X_N]$). Let $S: I\rightarrow\mathbb{R}$ be bounded and
piecewise continuous, where the only possible discontinuities are finite jumps occuring at the
points in $\{X_{1},X_{2},...,X_{N-1}\}$. Let
\[
s :=\max\{\left\vert S(x)\right\vert :x\in I\}.
\]

Denote by ${C} = {C}(I)\ $ the set of continuous functions
$f:I\rightarrow\mathbb{R}$. It is well-known that $({C},d_{\infty})$ is a
complete metric space, where
\[
d_{\infty}(f,g)=\max\{\left\vert f(x)-g(x)\right\vert :x\in I\}\text{.}%
\]
Let
\begin{align*}
{C}^{\ast}  &  :=\{f\in{C}:f(X_{0})=Y_{0},f(X_{N})=Y_{N}\},\\
{C}^{\ast\ast}  &  :=\{f\in{C}:f(X_{j})=Y_{j}\text{ for }j=0,1,...,N\}.
\end{align*}
Note that ${C}^{\ast}$ and ${C}^{\ast\ast}$ are closed subspaces of ${C}$
with ${C}^{\ast\ast}\subset{C}^{\ast}\subset{C}$. We say that each of the
functions in ${C}^{\ast\ast}$ \textit{interpolates the data} $\{(X_{j}%
,Y_{j})\st j = 0,1,\ldots, N\}$.

Let $b\in{C}^{\ast}$ and $h\in{C}^{\ast\ast}$. Define $T:{C^*}\rightarrow{C^{**}}$
by
\be\label{Top}
Tg := h+S\cdot(g\circ L-b\circ L).
\ee
$T$ is a form of Read-Bajraktarevi\'{c} operator as defined in
\cite{massopust}. The following result is a corrected version of \cite[Theorem
5.1, p. 136]{massopust}. See also \cite[Theorem 3, p. 731]{hutchinson}.

\begin{theorem}
\label{operatorthm} The mapping $T:{C^*}\rightarrow{C^{**}}$ obeys
\[
d_{\infty}(Tg_1,Tg_2)\leq s \,d_{\infty}(g_1,g_2)
\]
for all $g_1,g_2\in{C}^*$. In particular, if $s<1$ then $T$ is a contraction and it
possesses a unique fixed point $f\in{C}^{\ast\ast}.$
\end{theorem}

\begin{proof}
The operator $T$ is well-defined. Indeed, for $i = 1, \ldots, N-1$,
\[
Tg(X_i-) = h(X_i) = Tg(X_i+).
\]
To prove contractivity in the Chebyshev norm $\| \cdot \|_\infty$, observe that
\begin{align*}
d_{\infty}(Tg_1,Tg_2)  &  =\max\{\left\vert S(x)(g_1(L(x))-g_2(L(x)))\right\vert
:x\in I\}\\
&  \leq s\max\{\left\vert (g_1(l_{n}^{-1}(x))-g_2(l_{n}^{-1}(x)))\right\vert
:x\in\lbrack X_{n-1},X_{n}],n=1,2,...,N\}\\
&  =s\, d_{\infty}(g_1,g_2).
\end{align*}
The existence of a unique fixed point $f\in{C}$ (when $s<1)$ follows from the
contraction mapping theorem. Since $f({C}^{\ast})\subset{C}^{\ast\ast}$ and
$\left(  {C}^{\ast\ast},d_{\infty}\right)  $ is closed, hence complete$,$ it
follows that $f\in{C}^{\ast\ast}$.
\end{proof}

Note that $Tg=H+S\cdot g\circ L$ where $H=h-S\cdot b\circ L$. This tells us
that a fractal interpolation function $f$ is uniquely defined by three
functions $H,$ $S$, and $L$, of the special forms defined above.

The fixed point $f$ of $T$ interpolates the data $\left\{  (X_{j}%
,Y_{j}):j=0,1,2,...,N\right\}  $ and is an example of a fractal interpolation
function \cite{barninterp}. One way to evaluate $f$ is to use
\[
f=\lim_{k\rightarrow\infty}T^{k}(f_{0}),
\]
where $f_{0}\in{C}^{\ast}$. The proof of the contraction mapping theorem gives also an
estimate for the rate of convergence (cf. \cite{SB}], Theorem 5.2.3.):
\be\label{error}
\left\Vert f-T^{k}(f_{0})\right\Vert _{\infty}\leq \frac{s^{k}}{1-s}\left\Vert
f_1 - f_{0}\right\Vert _{\infty}.
\ee
In addition, an estimate for the operator $T$ can also be derived (cf. \cite{M2}):
\[
\|T\|_\infty \leq \frac{1+s}{1-s}.
\]

\section{\label{metricsec}The metric space $(I\times\mathbb{R},d_{q})$}

The following metric generalizes the ``taxi-cab" metric. We will need it in
the proof of Theorem \ref{ifsthm}.

\begin{proposition}\label{prop1}
\label{metricthm}Let $\alpha,\beta>0$ and $q:I\rightarrow\mathbb{R}$. Let
$d_{q}:\left(  I\times\mathbb{R}\right)  \times\left(  I\times\mathbb{R}%
\right)  \rightarrow\lbrack0,\infty)$ be defined by
\[
d_{q}(\left(  x_{1},y_{1}\right)  ,\left(  x_{2},y_{2}\right)  ):=\alpha
\left\vert x_{1}-x_{2}\right\vert +\beta\left\vert \left(  y_{1}%
-q(x_{1})\right)  -\left(  y_{2}-q(x_{2})\right)  \right\vert ,
\]
for all $(x_{1},y_{1})$, $(x_{2},y_{2})\in I\times\mathbb{R}.$ Then $d_{q}$ is
a metric on $I\times\mathbb{R}$. If $q$ is continuous then $\left(
I\times\mathbb{R},d_{q}\right)  $ is a complete metric space.
\end{proposition}

\begin{proof}
Clearly $d_{q}(\left(  x_{2},y_{2}\right)  ,\left(  x_{1},y_{1}\right)
)=d_{q}(\left(  x_{1},y_{1}\right)  ,\left(  x_{2},y_{2}\right)  )\geq0$.
Suppose that $d_{q}(\left(  x_{1},y_{1}\right)  ,\left(  x_{2},y_{2}\right)
)=0$. Then $\alpha\left\vert x_{1}-x_{2}\right\vert +\beta\left\vert \left(
y_{1}-q(x_{1})\right)  -\left(  y_{2}-q(x_{2})\right)  \right\vert =0$ which
implies $x_{1}=x_{2}$. Hence $\left\vert \left(  y_{1}-q(x_{1})\right)
-\left(  y_{2}-q(x_{1})\right)  \right\vert =0$ which implies $y_{1}=y_{2}$.

Demonstration that $d$ obeys the triangle inequality. Let $\left(  x_{i}%
,y_{i}\right)  \in I\times\mathbb{R}$, for $i=1,2,3.$ Write $q_{i}=q(x_{i})$
for $i=1,2,3$. We have%
\begin{align*}
&  d_{q}(\left(  x_{1},y_{1}\right)  ,\left(  x_{2},y_{2}\right)
)+d_{q}(\left(  x_{2},y_{2}\right)  ,\left(  x_{3},y_{3}\right)  )\\
&  =\alpha\left\vert x_{1}-x_{2}\right\vert +\beta\left\vert \left(
y_{1}-q_{1}\right)  -\left(  y_{2}-q_{2}\right)  \right\vert +\alpha\left\vert
x_{2}-x_{3}\right\vert +\beta\left\vert \left(  y_{2}-q_{2}\right)  -\left(
y_{3}-q_{3}\right)  \right\vert \\
&  =\alpha(\left\vert x_{1}-x_{2}\right\vert +\left\vert x_{2}-x_{3}%
\right\vert )+\beta(\left\vert \left(  y_{1}-q_{1}\right)  -\left(
y_{2}-q_{2}\right)  \right\vert +\left\vert \left(  y_{2}-q_{2}\right)
-\left(  y_{3}-q_{3}\right)  \right\vert )\\
&  \geq\alpha(\left\vert x_{1}-x_{3}\right\vert )+\beta(\left\vert \left(
y_{1}-q_{1}\right)  -\left(  y_{2}-q_{2}\right)  \right\vert +\left\vert
\left(  y_{2}-q_{2}\right)  -\left(  y_{3}-q_{3}\right)  \right\vert )\\
&  \geq\alpha(\left\vert x_{1}-x_{3}\right\vert )+\beta(\left\vert \left(
y_{1}-q_{1}\right)  -\left(  y_{3}-q_{3}\right)  \right\vert )=d_{q}(\left(
x_{1},y_{1}\right)  ,\left(  x_{3},y_{3}\right)  )\text{.}%
\end{align*}

To prove completeness in the case that $q$ is continuous, let $\left\{
(x_{k},y_{k})\right\}  _{k=1}^{\infty}$ denote a Cauchy sequence with respect
to the metric $d_{q}$. Given $\varepsilon>0$ we can find an integer
$N(\varepsilon)$ so that
\[
\alpha\left\vert x_{k}-x_{l}\right\vert +\beta\left\vert \left(  y_{k}%
-q(x_{k})\right)  -\left(  y_{l}-q(x_{l})\right)  \right\vert <\varepsilon
\]
whenever $k,l>N(\varepsilon)$. It follows that $\left\{  x_{k}\right\}  $ is a
Cauchy sequence with respect to the Euclidean norm, and so it converges, with
limit $x^{\ast}\in I$. Since $q$ is continuous, it now follows that $\left\{
q(x_{k})\right\}  $ converges to some limit $q^{\ast}\in\mathbb{R}$. In turn,
it follows that $\left\{  y_{k}\right\}  $ converges to some $y^{\ast}%
\in\mathbb{R}$. Hence $\left\{  (x_{k},y_{k})\right\}  _{k=1}^{\infty}$
converges to $(x^{\ast},y^{\ast})\in I\times\mathbb{R}$. It follows that
$\left(  I\times\mathbb{R},d_q\right)  $ is complete.
\end{proof}

\section{Fractal interpolants as attractors of iterated function systems}

Here we characterize the graph of the fixed point $f$ of $T$ as an attractor
of an IFS. Define $w_{n}:I\times\mathbb{R\rightarrow}I\times\mathbb{R}$ by%
\[
w_{n}(x,y) :=(l_{n}(x),h(l_{n}(x))+S(l_{n}(x))(y-b(x)))\text{.}%
\]
Define an IFS by%
\[
\mathcal{W}:=(I\times\mathbb{R};w_{1},w_{2},...,w_{N})\text{.}%
\]

Here we make use of the metric $d_{q}$ of Proposition \ref{metricthm} with $q=f$,
the fixed point of $T$. Let $\eta > 0$ and let
\[
\mathbb{X} :=\left\{  (x,y)\in I\times \R : \left\vert y-f(x)\right\vert \leq \eta\right\}
.
\]
It is readily verified that, when Theorem \ref{operatorthm} holds, namely when
$s<1,\mathcal{W}(\mathbb{X)} \subset\mathbb{X}$. The following theorem gives
conditions under which (i) the IFS $(\mathbb{X};w_{1},w_{2},...,w_{N})$ is
contractive with respect to $d_{f}$ and (ii) $\mathcal{W}$ has a unique
attractor. This result is a substantial generalization of \cite[Theorem 5.3,
p. 140]{massopust} which would require, in the present setting, that $h$ is
uniformly Lipschitz. Here, we avoid this restriction by using the metric
$d_{q}$ with $q=f$.

\begin{theorem}\label{thm4}
\label{ifsthm} Let $s<1$ and let $f\in{C}^{\ast\ast}$ be the fixed point of
$T,$ as in Theorem \ref{operatorthm}. Let $l_{n}:I\rightarrow I$\ have uniform
Lipschitz constant $\lambda_{l}<1$, such that $\left\vert l_{n}(x_{1}%
)-l_{n}(x_{2})\right\vert \leq\lambda_{l}\left\vert x_{1}-x_{2}\right\vert $
for all $x_{1},x_{2}\in I,$ for all $n$. Let $S:I\rightarrow\lbrack-s,s]$ have
Lipschitz constant $\lambda_{S}$, so that $\left\vert S(x_{1})-S(x_{2}%
)\right\vert \leq\lambda_{S}\left\vert x_{1}-x_{2}\right\vert $ for all
$x_{1},x_{2}\in I.$ Then the IFS $(\mathbb{X};w_{1},w_{2},...,w_{N})$ is
contractive with respect to the metric $d_{f}$ with $\alpha=1$ and
$0<\beta<\,\left(  1-\lambda_{l}\right)  /\lambda_{l}\lambda_S\eta$. In
particular, under these conditions, the IFS $\mathcal{W}$ has a unique
attractor $A=\Gamma(f)$, the graph of $f$, with basin of attraction
$I\times\mathbb{R}$.
\end{theorem}

\begin{proof}
Let $\left(  x_{1},y_{1}\right)  ,\left(  x_{2},y_{2}\right)  \in\mathbb{X}$.
We have
\begin{align*}
&  d_{f}(w_{n}\left(  x_{1},y_{1}\right)  ,w_{n}\left(  x_{2},y_{2}\right)
)-\alpha\left\vert l_{n}(x_{1})-l_{n}(x_{2})\right\vert \\
&  =\beta|h(l_{n}(x_{1}))+S(l_{n}(x_{1}))(y_{1}-b(x_{1}))-f(l_{n}(x_{1}))\\
&  - (h(l_{n}(x_{2}))+S(l_{n}(x_{2}))(y_{2}-b(x_{2}))-f(l_{n}(x_{2})))|\\
&  =\beta|\left(  S(l_{n}(x_{1}))(y_{1}-f(x_{1}))\right)  -\left(
S(l_{n}(x_{2}))(y_{2}-f(x_{2}))\right)  |\\
&  \leq\beta\left\vert S(l_{n}(x_{1}))\right\vert \cdot|(y_{1}-f(x_{1}%
))-(y_{2}-f(x_{2}))|\\
&  +|S(l_{n}(x_{1}))-S(l_{n}(x_{2}))|\cdot\left\vert (y_{2}-f(x_{2}%
))\right\vert \\
&  \leq\beta s|(y_{1}-f(x_{1}))-(y_{2}-f(x_{2}))|+\beta\lambda_{l}\lambda
_{S}\eta\left\vert x_{1}-x_{2}\right\vert .
\end{align*}
To obtain the second equality, we used the fact that $f$ is the fixed point of \eqref{Top}. 

Hence%
\begin{align*}
&  d_{f}(w_{n}\left(  x_{1},y_{1}\right)  ,w_{n}\left(  x_{2},y_{2}\right)
)\\
&  \leq\left(  \alpha\lambda_{l}+\beta\lambda_{S}\lambda_{l}\eta\right)
\left\vert x_{1}-x_{2}\right\vert +\beta s|(y_{1}-f(x_{1}))-(y_{2}%
-f(x_{2}))|\\
&  \leq\left(  \alpha+\beta\lambda_{S}\eta\right)  \lambda_{l}\left\vert
x_{1}-x_{2}\right\vert +\beta s|(y_{1}-f(x_{1}))-(y_{2}-f(x_{2}))|\\
&  \leq c\cdot\left(  \alpha\left\vert x_{1}-x_{2}\right\vert +\beta\left\vert
\left(  y_{1}-f(x_{1})\right)  -\left(  y_{2}-f(x_{2})\right)  \right\vert
\right)
\end{align*}
where $c:=\max\left\{  s,\lambda_{l}+\beta\lambda_{l}\lambda_{S}\eta%
/\alpha\right\}  .$ Since $\lambda_{l}<1$ we can choose $\alpha,\beta>0$ so
that $c<1.$ For example, we can choose $\alpha=1$ and $0<\beta<\,\left(
1-\lambda_{l}\right)  /\lambda_{l}\lambda_{S}\eta$.

It follows that the IFS $\widetilde{\mathcal{W}}:=(\mathbb{X};w_{1}%
,w_{2},...,w_{N})$ is contractive, and hence it has a unique attractor. This
attractor must be $\Gamma(f)$ because a contractive IFS has a unique nonempty
compact invariant set and it is readily verified that $\widetilde{\mathcal{W}%
}(\Gamma(f))=\Gamma(f)$. Since we can choose the constant $\eta$ arbitrarily
large, it now follows that $\mathcal{W}$ has a unique attractor, namely
$\Gamma(f)$. Note that we have {\em not} provided a metric with respect to which
$\mathcal{W}$ is contractive!
\end{proof}

We remark that
\[
\Gamma(Tg)=\mathcal{W}\left(  \Gamma(g)\right)  , \text{ for all }g\in
{C^*}\text{.}%
\]
When, for example, $S$ is Lipschitz continuous with Lipschitz constant $s<1$, and the functions $l_{n}$ are
contractive, the graph of the fractal interpolant $f$ can be approximated by
the ``chaos game" algorithm. (See \cite{BL} and \cite{BWL} for new topological
viewpoints of the ``chaos game.")

\section{Bilinear fractal interpolation}

We consider a specific example of the preceding theory. Let $l_n:I\to [X_{n-1},X_n]$ be given by
\be\label{ln}
l_{n}(x) :=X_{n-1}+ \left(  \frac{X_{n}-X_{n-1}}{X_{N}-X_{0}}\right)
(x-X_{0})
\ee
and $S:I\to \R$ by
\[
S :=S_{n}\circ l_{n}^{-1},
\]
for $x\in\lbrack X_{n-1},X_{n}],$ $n = 1, \ldots, N$, where $S_n :I \to \R$,
\[
S_{n}(x):=s_{n-1}+\left(  \frac{s_{n}-s_{n-1}}{X_{N}-X_{0}}\right)  \left(
x-X_{0}\right)  \text{,}%
\]
with $\left\{  s_{j}:j=0,1,2,...,N\right\}  \subset\mathbb{(-}1,1)$. Note that the $s_j$, $j = 0,1,\ldots, N$, need not be ordered. 

Then $S$ is continuous and%
\begin{align*}
\left\vert S(x)\right\vert  &  \leq\max\{|S_{n}(l_{n}^{-1}(x))|:x\in\lbrack
X_{n-1},X_{n}],n\in\left\{  1,2,...,N\right\}  \}\\
&  =\max\left\{  \left\vert s_{j}\right\vert :j=0,1,...,N\right\}  = :s<1.
\end{align*}
Furthermore, let $b:I\to\R$ be given by
\be\label{b}
b(x):=Y_{0}+\left(  \frac{Y_{N}-Y_{0}}{X_{N}-X_{0}}\right)  (x-X_{0})
\ee
and $h:I\to \R$ by
\be\label{h}
h(x):= \sum_{n=1}^N \left[Y_{n-1}+\left(  \frac{Y_{n}-Y_{n-1}}{X_{n}-X_{n-1}}\right)  (x-X_{n-1})\right]\chi_{[X_{n-1},X_n]}(x),
\ee
where $\chi_M$ denotes the characteristic function of a set $M$.

Note that $b\in C^*$ and $h\in C^{**}$. Theorem \ref{operatorthm} implies that $T$ has a unique fixed point $f$.
Specifically, $f$ is the unique solution of the set of functional equations of the form
\be\label{functeq}
f(l_{n}(x))-h(l_{n}(x))=S_n(l_n(x))[f(x)-b(x)],\quad n = 1, \ldots, N;\; x\in I.
\ee
We refer to $f$ as a \textbf{bilinear fractal interpolant}. The reason for
this name is that in this case the functions $w_{n}$ of the IFS $\mathcal{W}$
take the form%
\[
w_{n}(x,y) :=(l_{n}(x),a+bx+cy+dxy),
\]
where $a,b,c,d$ are real constants. Functions of the form $B: (x,y)\mapsto
a+bx+cy+dxy$ are called \textit{bilinear} in the computer graphics literature. We will adhere to this terminology but like to point out that $B$ is for fixed $x$ or fixed $y$ affine in the other variable. More precisely,
\begin{gather*}
B((1-t)x_1 + t x_2, y) = (1-t)B(x_1,y) + t B(x_2,y)\\
B(x, (1-t)y_1 + ty_2) = (1-t)B(x,y_1) + t B(x,y_2),
\end{gather*}
for all $x_1, x_2, y_1, y_2, t\in \R$.

Using the expressions for $l_n$, $S_n$, and $h$ above, we can write the functions $w_n$ in the form
\begin{align}
w_{n} &  (x,y)= \left(X_{n-1}+ \left(  \frac{X_{n}-X_{n-1}}{X_{N}-X_{0}}\right)
(x-X_{0}),Y_{n-1}+\left(  \frac{Y_{n}-Y_{n-1}}{X_{N}-X_{0}%
}\right)  (x-X_{0})\right.\label{w}\\
& \left. + \left[  s_{n-1}+\left(  \frac{s_{n}-s_{n-1}}{X_{N}-X_{0}}\right)
(x-X_{0})\right]  \left[  y-Y_{0}-\left(  \frac{Y_{N}-Y_{0}}{X_{N}-X_{0}%
}\right)  (x-X_{0})\right]\right)  .\nonumber
\end{align}
In particular note that
\[
w_{n}(X_{N},y)=(X_{n},Y_{n}+s_{n}(y-Y_{N}))\;\;\text{and}\;\;w_{n+1}%
(X_{0},y)=(X_{n},Y_{n}+s_{n}(y-Y_{0})).
\]
It follows that the images of any (possibly degenerate) parallelogram with
vertices at $(X_{0},Y_{0}\pm H)$ and $(X_{N},Y_{N}\pm H)$, for $H\in
\mathbb{R}$ under the IFS fit together neatly, as illustrated in Figure
\ref{interp-b}.%
\begin{figure}[ptb]%
\centering
\includegraphics[
height = 7.5cm,
width = 12.8cm
]%
{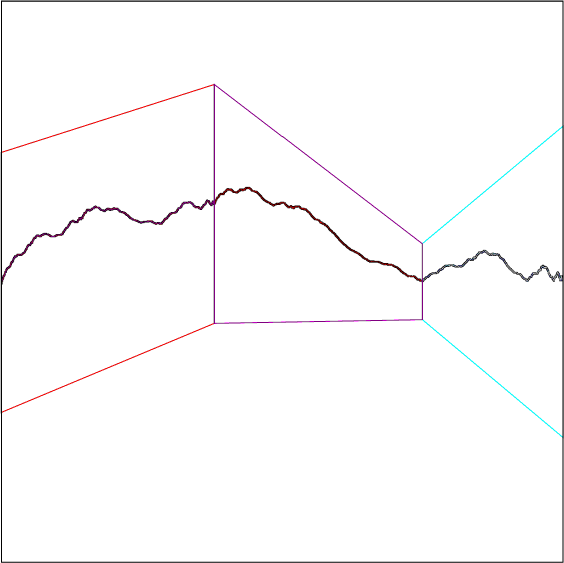}%
\caption{A fractal interpolation function defined by three bilinear
transformations. See text.}%
\label{interp-b}%
\end{figure}

\section{Box dimension of bilinear interpolants}

In this section, we derive a formula for the box dimension of the graphs of a class of bilinear interpolants. To this end, let $1 < N\in \N$, let $I := [0, 1]\subset \R$ be the unit interval, and let $\blacksquare := I\times I$ denote the filled unit square. Suppose that $\{0 =: X_0 < X_1 < \cdots < X_N := 1\}$ is a set of knots in $I$. Furthermore, suppose that $\{\uY_j \in I \st j = 0, 1, \ldots, N\}$ and $\{\oY_j \in I \st j = 0, 1, \ldots, N\}$ are two sets of points with the property that $0\leq \uY_j \leq \oY_j < 1$, $\forall j = 0, 1, \ldots, N$.

Denote by $Q_n$ the trapezoid with vertices $A_n := (X_{n-1}, \uY_{n-1})$, $B_n := (X_n, \uY_n)$,
$C_n := (X_n, \oY_n)$, and $D_n := (X_{n-1}, \oY_{n-1})$, $n = 1, \ldots, N$. For each $n = 1, \ldots, N$,
let $l_n : I \to  [X_{n-1},X_n]$ be a family of affine mappings and $B_n : I \times \R \to \R,
(x, y) \mapsto a_n x + b_n y + c_n x y + d_n = (d_n + a_n x) + (b_n + c_n x) y$, $a_n, b_n, c_n, d_n \in \R$, a
family of bilinear mappings.

Define mappings $w_n := (l_n, B_n) : \blacksquare \to Q_n$ by requiring that
\be\label{eq1}
(0,0)\stackrel{w_n}{\longmapsto} A_n,\quad (1,0) \stackrel{w_n}{\longmapsto} B_n,\quad (1,1) \stackrel{w_n}{\longmapsto} C_n,\quad (0,1) \stackrel{w_n}{\longmapsto} D_n.
\ee
It follows readily from (\ref{eq1}) that the affine mappings $l_n$ are given by \eqref{ln} and the bilinear mappings $B_n$ by
\be\label{Bn}
B_n (x,y) = a_n x + [s_{n-1} + (s_n - s_{n-1}) x] \,y + \uY_{i-1} ,
\ee
where we set $a_n = \uY_n - \uY_{n-1}$, $n = 1, \ldots, N$, and $s_j := \oY_j - \uY_j$, $j = 0,1,\ldots, N$. Note that $0 \leq s_j < 1$, for all $j = 0,1,\ldots, N$.

\begin{definition}
The IFS $\cF := (\blacksquare; w_1, \ldots, w_N)$ where $w_n := (l_n, B_n)$ with $l_n$ and $B_n$, $n=1, \ldots, N$, given by \eqref{ln} and \eqref{Bn}, respectively, is called \textbf{bilinear}.
\end{definition}

In \cite{BV} such bilinear IFSs are investigated in more generality and in connection with fractal homeomorphisms. The approach undertaken in \cite{BV} makes substantial use of the geometric properties that functions in an bilinear IFS possess, namely that they take horizontal and vertical lines to lines and that they preserve proportions along horizontal and vertical lines. For further details and results, we refer the interested reader to \cite{BV}.

Recall the definition of the metric $d_q$ given in Proposition \ref{prop1}. For our current purposes, we set $q \equiv 1$. As in Theorem \ref{thm4} we denote the Lipschitz constant of the $l_n$ by $\lambda_n$.

\begin{theorem}
The bilinear IFS $\cF = (\blacksquare; w_1, \ldots, w_N)$ is contractive in the metric $d_1$ with $\alpha := 1$ and $0 < \beta < \frac{1- \lambda_n}{2}$.
\end{theorem}
\begin{proof}
It suffices to show that each $w_n\in \cF$ is contractive with respect to the metric $d_1$. To this end, let $(x,y), (x',y')\in\blacksquare$ and set $\Delta s_n:= s_n - s_{n-1}$, $n = 1, \ldots, N$. Then
\begin{align*}
d_1(&w_n(x,y),w_n(x',y')) = |l_n(x) - l_n(x')| + \beta |B_n(x,y) - B_n(x',y')|\\
& \leq \lambda_n |x - x'| + \beta (|a_n| |x - x'| + |[s_{n-1}+ \Delta s_n x] y - [s_{n-1}+ \Delta s_n x'] y'|)\\
& = \lambda_n |x - x'| + \beta(|a_n| |x - x'| + |[s_{n-1}+ \Delta s_n x] y\\
& \qquad - |[s_{n-1}+ \Delta s_n x'] y + |[s_{n-1}+ \Delta s_n x'] y -|[s_{n-1}+ \Delta s_n x'] y')\\
& \leq \lambda_n |x - x'| + \beta(|a_n| |x - x'| + |\Delta s_n y| |x - x'| + |s_{n-1} + \Delta s_n x'| |y - y‘|)\\
& \leq \left(\lambda_n + \beta (|a_n| + |\Delta s_n|\right) |x - x'| + \beta |s_n| |y - y'| \quad \text{(since $|x'|, |y| \leq 1$)}\\
& \leq \max\left\{\lambda_n + 2\beta, s_n\right\}\,d_1((x,y),(y',y'))\quad \text{(since $|a_n|, |\Delta s_n| \leq 1$)}.
\end{align*}
By choosing $0 < \beta < \frac{1-\lambda_n}{2}$ the maximum can be made strictly smaller than 1.
\end{proof}

Adapting \eqref{ln}, \eqref{b}, and \eqref{h} to the current setting using $\{\uY_j\st j = 0,1,\ldots, N\}$ instead of $\{Y_j\st j = 0,1,\ldots, N\}$, we see by Theorem \ref{operatorthm} that the associated operator $T :C^* \to C^{**}$ defined by
\be\label{T}
Tg := h + [s_{n-1} + \Delta s_n\,(\bullet)]\cdot (g - b)\circ L
\ee
is contractive and its unique fixed point $f$ is an element of $C^{**}$. Moreover, $f$ satisfies the functional equations set forth in \eqref{functeq}.

Next, we derive a formula for the box dimension of the graphs of bilinear fractal interpolants arising from the above bilinear IFS $\cF = (\blacksquare; w_1, \ldots, w_N)$. For this purpose, we may assume, without loss of generality, that $\uY_0 = \uY_N = 0$. This special case can always be achieved by means of an affine transformation (which does not change the box dimension).

To this end, we recall the definition of box-counting or box dimension of a bounded set $M\subset\mathbb{R}^{n}$:
\begin{equation}
\dim_{B} M :=\lim_{\varepsilon\rightarrow0+}\frac{\log\mathcal{N}%
_{\varepsilon}(M)}{\log\varepsilon^{-1}}, \label{dimeq}%
\end{equation}
where $\mathcal{N}_{\varepsilon}(M)$ is the minimum number of square boxes
with sides parallel to the axes, whose union contains $M.$ By the statement
``$\dim_{B}M=D$" we mean that the limit in equation (\ref{dimeq}) exists and
equals $D.$

In the case where $M$ is the graph $\Gamma(f)$ of a function $f$, knowledge of
the box dimension of $\Gamma(f)$ provides information about the smoothness of
$f$ since $\dim_{B}\Gamma(f)$ is related to H\"{o}lder exponents associated
with $f$. (See, for example, \cite[Section 12.5]{tricot}.)

The following result gives an explicit formula for the box dimension of the graph
of a bilinear fractal interpolant defined via the operator \eqref{T}. The proof is based on arguments first
applied in \cite{HM}.

\begin{theorem}
\label{dimthm} Let $\mathcal{F}$ denote the bilinear IFS defined above and let $\Gamma(f)$ denote its attractor. Suppose that the knots $\{X_j\st j = 0,1, \ldots, N\}$ are uniformally spaced on $I$, i.e., $X_j = j/N$, $\forall j = 0,1,\ldots, N$, and suppose that $s_0 = s_N$. If $\gamma := {\sum_{n=1}^{N}}\,\frac{s_{n-1}+s_{n}}{2} > 1$ and $\Gamma(f)$ is not a straight line segment then%
\[
\dim_{B}\Gamma(f)=1+\frac{\log{\gamma}}{\log N};
\]
otherwise $\dim_{B}\Gamma(f)=1.$
\end{theorem}

\begin{proof}
Note that in the computation of the box dimension of $\Gamma(f)$ it suffices
to consider covers of $\Gamma(f)$ whose elements are squares of side $N^{-r}$,
$r\in\mathbb{N}_{0} := \N\cup\{0\}$. Denote by $\mathcal{C}_{0} (r)$ a cover of $\Gamma(f)$
consisting of a finite number of squares of side $N^{-r}$, $r\in\mathbb{N}%
_{0}$. Now consider a specific cover $\mathcal{C} (r)$ of $\Gamma(f)$ of the
form
\begin{equation}
\label{cover}\mathcal{C} (r) := \left\{  \left[  \frac{k-1}{N^{r}},\frac
{k}{N^{r}}\right]  \times\left[  a,a+\frac{1}{N^{r}}\right]  : r\in
\mathbb{N}_{0};\,k = 1, \ldots, N^{r};\,a \in\mathbb{R}\right\}  .
\end{equation}
By the compactness of $\Gamma(f)$, there exists a minimal cover $\mathcal{C}%
^{*}_{0} (r)$ of $\Gamma(f)$ and also a minimal cover $\mathcal{C}^{*} (r)$ of
$\Gamma(f)$ of the form (\ref{cover}). Denote by $\mathcal{N}_{0} (r)$,
respectively, $\mathcal{N} (r)$ the cardinality of these minimal covers. Since
covers of the form (\ref{cover}) are more restrictive, we have $\mathcal{N}%
_{0} (r) \leq\mathcal{N} (r)$. On the other hand, every $(N^{-r}\times
N^{-r})$-square in $\mathcal{C}^{*}_{0} (r)$ can be covered by at most two
$(N^{-r}\times N^{-r})$-squares from a cover of the form (\ref{cover}). Thus,
$\mathcal{N} (r) \leq2 \mathcal{N}_{0} (r)$. Hence, when computing the box
dimension of $\Gamma(f)$ it suffices to consider covers of the form
(\ref{cover}).

To this end, let $r\in\mathbb{N}_{0}$ be fixed. Let $\mathcal{C} (r)$ be a
minimal cover of $\Gamma(f)$ of cardinality $\mathcal{N}(r)$ consisting of
squares of side $N^{-r}$ whose interiors are disjoint. Let $\mathcal{C} (r,k)$
be the collection of all squares in $\mathcal{C} (r)$ that lie between $x =
\frac{k-1}{N^{r}}$ and $x = \frac{k}{N^{r}}$, $k = 1, \ldots, N^{r}$. Denote
by $\mathcal{N}(r,k)$ the cardinality of $\mathcal{C}(r,k)$, and let
\[
\mathcal{R} (r,k) := \bigcup_{C\in\mathcal{C}(r,k)} C.
\]
As $\mathcal{C} (r)$ is a cover of $\Gamma(f)$ of minimal cardinality, every
square in $\mathcal{C}(r)$ must meet $\Gamma(f)$, and since $f$ is continuous
on $I$, the set $\mathcal{R}(r,k)$ must be a rectangle of width $N^{-r}$ and
height $N^{-r} \mathcal{N}(r,k)$. Note that $\mathcal{N}(r) =
\sum_{k=1}^{N^{r}}\;\mathcal{N}(r,k)$.

Now apply the mappings $w_{n}$, $n=1, \ldots, N$, defined in (\ref{eq1}) to the
rectangle $\mathcal{R}(r,k)$. The image of $\mathcal{R}(r,k)$ under $w_{n}$ is
a trapezoid contained in the strip $\displaystyle{\left[  \frac{l(k,n)-1}%
{N^{r+1}},\frac{l(k,n)}{N^{r+1}}\right]  \times\mathbb{R}}$, with $l (k,n) :=
k + (n-1) N^{r}$. Observe that
\[
\mathcal{N}(r+1) = \sum_{n=1}^{N} \sum_{k=1}^{N^{r}} \mathcal{N}(r+1,l(k,n)).
\]
The fixed point equation for $\Gamma(f)$, namely, $\Gamma(f) =
\displaystyle{\bigcup_{n=1}^{N}\,w_{i}} (\Gamma(f))$, implies that
\[
\Gamma(f) \subseteq\bigcup_{n=1}^{N} w_{n} \left(  \bigcup_{k=1}^{N^{r}}
\mathcal{R}(r,k)\right)  .
\]
Depending on the sign of $\Delta s_n$, there are ten possible geometric shapes for the trapezoid $w_{n} (\mathcal{R}(r,k))$. In Figure \ref{fig2} one of these trapezoids is depicted and the relevant geometric quantities identified. 
\begin{figure}[h!]%
\centering
\includegraphics[height = 4cm,width = 8cm]{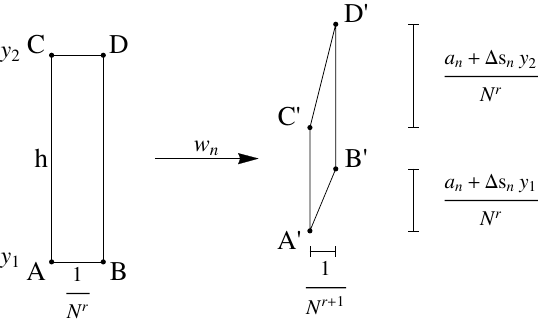}%
\caption{An image of a rectangle under the map $w_n$.}%
\label{fig2}%
\end{figure}

Employing the notation in Figure \ref{fig2}, we write $A' < B'$ if the $y$--coordinate of the point $A'$ is less than the $y$--coordinate of the point $B'$. Similarly, we define $A'\leq B'$.
\nl
\textbf{Case I: $\Delta s_n \geq 0$.}
Note that in this scenario, distance($A'$, $C'$) $\leq$  distance($B'$, $D'$). The five possible shapes are given by the location of the vertices $A'$, $B'$, $C'$, and $D'$. They are: $B' < D' \leq A' < C'$, $B' < A' \leq D' \leq C'$, $B' \leq A' < C'\leq D'$, $A'\leq B' \leq C' < D'$, and $A' < C' \leq B' < D'$. Each one of these trapezoids is contained in a rectangle of width $N^{-(r+1)}$ and height at most 
\be\label{upper_pos}
\left(s_{n-1} + \Delta s_n \cdot \frac{k}{N^r}\right) (N^{-r} \cN(r,k)) + \frac{2(|a_n| + |\Delta s_n|)}{N^r},
\ee
and meets a rectangle of width $N^{-(r+1)}$ and height at least 
\be\label{lower_pos}
\left(s_{n-1} + \Delta s_n \cdot \frac{k-1}{N^r}\right) (N^{-r} \cN(r,k)) - \frac{2(|a_n| + |\Delta s_n|)}{N^r}.
\ee
Hence, 
\begin{align}
\mathcal{N}(r+1,l(k,n)) &\leq \left[\left(s_{n-1} + \Delta s_n \cdot \frac{k}{N^r}\right) (N^{-r} \cN(r,k)) + \frac{2(|a_n| + |\Delta s_n|)}{N^r} \right] N^{r+1} + 2\nonumber\\
& = N \left(s_{n-1} + \Delta s_n \cdot \frac{k}{N^r}\right)\,\mathcal{N}(r,k) + 2N (|a_{n}| + |\Delta s_n|) + 2,
\end{align}
and, similarly, 
\begin{align}\label{7.9}
\mathcal{N}(r+1,l(k,n)) \geq N \left(s_{n-1} + \Delta s_n \cdot \frac{k-1}{N^r}\right)\,\mathcal{N}(r,k) - 2N (|a_{n}| + |\Delta s_n|) - 2.
\end{align}
\nl
\textbf{Case II: $\Delta s_n \leq 0$.}
Here, distance($A'$, $C'$) $\geq$  distance($B'$, $D'$) and the five possible shapes are as above. Each one of these trapezoids is contained in a rectangle of width $N^{-(r+1)}$ and height at most 
\be\label{upper_neg}
\left(s_{n-1} + \Delta s_n \cdot \frac{k-1}{N^r}\right) (N^{-r} \cN(r,k)) + \frac{2(|a_n| + |\Delta s_n|)}{N^r},
\ee
and meets a rectangle of width $N^{-(r+1)}$ and height at least 
\be\label{lower_neg}
\left(s_{n-1} + \Delta s_n \cdot \frac{k}{N^r}\right) (N^{-r} \cN(r,k)) - \frac{2(|a_n| + |\Delta s_n|)}{N^r}.
\ee
Thus, similar to Case I, we obtain an upper, respectively lower, bound for $\mathcal{N}(r+1,l(k,n))$ of the form
\begin{align}
\mathcal{N}(r+1,l(k,n)) \leq N \left(s_{n-1} + \Delta s_n \cdot \frac{k-1}{N^r}\right)\,\mathcal{N}(r,k) + 2N (|a_{n}| + |\Delta s_n|) + 2,
\end{align}
and
\begin{align}\label{7.13}
\mathcal{N}(r+1,l(k,n)) \geq N \left(s_{n-1} + \Delta s_n \cdot \frac{k}{N^r}\right)\,\mathcal{N}(r,k) - 2N (|a_{n}| + |\Delta s_n|) - 2.
\end{align}

Denote by $N_\pm$ the set of all indices $n\in\{1, \ldots, N\}$ for which $\Delta s_n \geq 0$, respectively, $\Delta s_n \leq 0$. Then, using Equations \eqref{upper_pos} and \eqref{upper_neg}, summation over $n$ yields
\begin{align*}
\sum_{n=1}^N \mathcal{N}(r+1,l(k,n)) &= \sum_{n\in N_+} \mathcal{N}(r+1,l(k,n)) + \sum_{n\in N_-} \mathcal{N}(r+1,l(k,n))\\
&\leq N \sum_{n\in N_+}\left(s_{n-1} + \Delta s_n \cdot \frac{k}{N^r}\right)\,\mathcal{N}(r,k)\\
& + N \sum_{n\in N_-} \left(s_{n-1} + \Delta s_n \cdot \frac{k-1}{N^r}\right)\,\mathcal{N}(r,k) \\
& +  \sum_{n=1}^N \left[2N (|a_{n}| + |\Delta s_n|) + 2\right]
\end{align*}
Now, 
\[
s_{n-1} + \Delta s_n \cdot \frac{k}{N^r} = \frac{s_{n-1}+s_n}{2} + \Delta s_n \left(\frac{k}{N^r} - \frac12\right), \quad n\in N_+,
\]
and
\begin{align*}
s_{n-1} + \Delta s_n \cdot \frac{k-1}{N^r} &= \frac{s_{n-1}+s_n}{2} + \Delta s_n \left(\frac{k-1}{N^r} - \frac12\right)\\ 
& = \frac{s_{n-1}+s_n}{2} + \Delta s_n \left(\frac{k}{N^r} - \frac12\right) + \frac{-\Delta s_n}{N^r},\quad n\in N_-.
\end{align*}
Substitution into the expression for $\sum_{n=1}^N \mathcal{N}(r+1,l(k,n))$ gives
\begin{align*}
\sum_{n=1}^N \mathcal{N}(r+1,l(k,n)) & \leq N \sum_{n=1}^N \left(\frac{s_{n-1}+s_n}{2}\right) \cN(r,k) + N \left(\sum_{n=1}^N \Delta s_n\right) \left(\frac{k}{N^r} - \frac12\right) \cN(r,k)\\
& + \sum_{n\in N_-} \frac{-\Delta s_n}{N^r} +  \sum_{n=1}^N \left[2N (|a_{n}| + |\Delta s_n|) + 2\right].
\end{align*}
As $\sum_{n=1}^N \Delta s_n = s_N - s_0 = 0$ by assumption, we obtain
\begin{align}\label{7.14}
\sum_{n=1}^N \mathcal{N}(r+1,l(k,n)) & \leq (N \gamma)\, \cN(r,k) + \sum_{n\in N_-} \frac{-\Delta s_n}{N^r} +  \sum_{n=1}^N \left[2N (|a_{n}| + |\Delta s_n|) + 2\right]\nonumber\\
& \leq (N \gamma)\, \cN(r,k) + c_1,
\end{align}
where we set $c_1 := \sum_{n=1}^N \left[2N (|a_{n}| + |\Delta s_n|) + \frac{|\Delta s_n|}{N}+ 2\right]$.

Summing \autoref{7.14} over $k$ produces an upper bound for $\cN(r+1)$ in terms of $\cN(r)$:
\[
\cN(r+1) \leq (N\gamma)\,\cN(r) + c_1 N^r.
\]
Induction on $r$ yields
\[
\cN(r) \leq (N\gamma)^r \,\cN(0) + c_1 N^{r-1} \sum_{\varrho = 0}^{r-1} \gamma^\varrho.
\]

Depending on the value of $\gamma$, two cases need to be considered.
\nl
\textbf{Case A: $\gamma\leq1$.} This implies that
$\mathcal{N}(r)\leq N^{r} (\mathcal{N}(0) + c_{1}r)$. Hence,
\[
\dim_{B} \Gamma(f) \leq\lim_{r\to\infty} \frac{\log N^r (\cN(0) + c_1 r)}{\log N^r} = 1.
\]
\nl\textbf{Case B: $\gamma> 1$.} Observing that in this situation
\[
\sum_{\varrho = 0}^{r-1} \gamma^\varrho \leq \frac{\gamma^r}{\gamma - 1},
\]
we obtain 
\[
\mathcal{N}(r)  \leq(\gamma N)^{r}\,\mathcal{N}(0) + \frac{c_1(N\gamma)^r}{\gamma - 1} =: c_2(N\gamma)^r.
\]
Thus,
\[
\dim_{B} \Gamma(f) \leq\lim_{r\to\infty} \frac{\log c_{2}\,(\gamma N)^{r}%
}{\log N^{r}} = 1 + \frac{\log\gamma}{\log N}.
\]

Note that since $f$ is a continuous function, $\dim_{B} \Gamma(f) \geq1$. If $\Gamma(f)$ is a
line segment, i.e., if the set of data $\mathcal{J}:= \{(j/N, \uY_{j}) \,:\, j
= 0,1,\ldots, N)\}$ is collinear, then $\Gamma(f) = [0,1]$ implying that $\dim_{B}
\Gamma(f) = 1$.

To obtain a nontrivial lower bound for $\Gamma(f)$, the following lemma is required.

\begin{lemma}\label{lemmadimension} 
If $\gamma:= \sum_{i=1}^{N}\, \frac{s_{i-1}+s_{i}}{2} > 1$, $s_0 = s_N$, and $\Gamma(f)$ is not a line segment then
\[
\lim_{r\to\infty}\frac{\cN(r)}{N^r} = \infty.
\]

\end{lemma}

\begin{proof}
The assumption that $\Gamma(f)$ is not a line segment implies the existence
of at least one index $n_{0}\in\{1, \ldots, N-1\}$ so that
\[
\delta:= \uY_{n_0} > 0.
\]
Since $f$ is continuous on $I$, we have that $\mathcal{N} (r) \geq\delta
N^{r}$. Note that $I$ is mapped to the line segments $\overline{(X_{n-1},\uY_{n-1}),(X_{n},\uY_{n})}$,
implying that for $r\geq 1$
\begin{align*}
\mathcal{N}(r) & \geq\sum_{n=1}^{N} \left[  s_{n-1} + \Delta s_{n}\,\frac{n_0}{N}\right]  \, \delta\, N^{r}\\
& = \left[\sum_{n=1}^{N} \left(\frac{s_{n-1}+s_n}{2}\right) + \sum_{n=1}^{N} \left(\frac{n_0}{N} - \frac12\right)\Delta s_n\right]\,\delta N^r\\
& = \sum_{n=1}^{N} \left(\frac{s_{n-1}+s_n}{2}\right) (\delta N^r).\qquad \text{(As the sum over $\Delta s_n$ equals zero.)}
\end{align*}
Proceeding inductively, we arrive at
\begin{align*}
\mathcal{N}(r) & \geq \sum_{n_{1},\cdots, n_{k}=1}^{N} \prod_{\ell=1}^{k} \left[\frac{s_{n_\ell - 1} + s_{n_\ell}}{2} \right] (\delta\, N^{r}),\qquad r\geq k.
\end{align*}
Therefore,
\[
\mathcal{N}(r) \geq [\gamma^{r}\, \delta -1] N^{r},
\]
which, since $\gamma> 1$, finishes the proof of the lemma.
\end{proof}

Suppose then that $\gamma> 1$ and that $\Gamma(f)$ is not a line segment,
i.e., $\mathcal{J}$ is not collinear. Since each $C\in\mathcal{C}(r,k)$
meets $\Gamma(f)$, the image of $C$ under the maps $w_{n}$, $n=1, \ldots,
N$, must also meet $\Gamma(f)$. 

Thus, using Equations \eqref{7.9} and \eqref{7.13}, we obtain
\begin{align*}
\sum_{n=1}^N \mathcal{N}(r+1,l(k,n)) &= \sum_{n\in N_+} \mathcal{N}(r+1,l(k,n)) + \sum_{n\in N_-} \mathcal{N}(r+1,l(k,n))\\
&\geq N \sum_{n\in N_+}\left(s_{n-1} + \Delta s_n \cdot \frac{k-1}{N^r}\right)\,\mathcal{N}(r,k)\\
& + N \sum_{n\in N_-} \left(s_{n-1} + \Delta s_n \cdot \frac{k}{N^r}\right)\,\mathcal{N}(r,k) \\
& -  \sum_{n=1}^N \left[2N (|a_{n}| + |\Delta s_n|) + 2\right]
\end{align*}
Algebra similar to that applied in the estimate for the upper bound, yields
\begin{align*}
\sum_{n=1}^N \mathcal{N}(r+1,l(k,n)) & \geq N \sum_{n=1}^N \left(\frac{s_{n-1}+s_n}{2}\right) \cN(r,k) \\
& - \sum_{n\in N_+} \frac{\Delta s_n}{N^r} -  \sum_{n=1}^N \left[2N (|a_{n}| + |\Delta s_n|) + 2\right].
\end{align*}
Summation over $k$ gives 
\[
\cN(r+1) \geq (N\gamma)\,\cN(r) - c_1 N^r.
\]
Hence,
\begin{align*}
\cN(r) & \geq (N\gamma)^{r-m} \cN(m) - c_1 N^{r-1} \sum_{\varrho = 0}^{r-m-1} \gamma^\varrho\\
& \geq (N\gamma)^{r-m}\left[\cN(m) - \frac{c_1 N^{m-1}}{1-\gamma^{-1}}\right],
\end{align*}
for all $m\in\mathbb{N}$ with $1\leq m \leq r$.

Lemma \autoref{lemmadimension} implies that we can choose $r$ and $m$ large enough so
that
\[
\mathcal{N}(m) - \frac{c_{1} N^{m-1}}{1-\gamma^{-1}} > 0.
\]
Therefore, $\mathcal{N}(r) \geq c_{2}\,(\gamma\,N)^{r}$, for a constant
$c_{2} > 0$ and for large enough $r$. Hence, $\dim_{B} \Gamma(f) \geq1 + \displaystyle{\frac{\log\gamma}{\log N}}$.
\end{proof}

\begin{remark}
Recall that the code space associated with an IFS is given by $\Omega := \{1, \ldots, N\}^\infty$. The elements of $\Omega$ are called codes. The set of all finite codes is defined as $\Omega^\prime := \bigcup_{k=0}^\infty \{1, \ldots, N\}^k$, where the empty set represents a code of length zero.

The proof of Theorem \ref{dimthm} shows in particular that for a given $\sigma= \sigma_{1}\ldots\sigma_{|\sigma|}\in \Omega^\prime$ of finite length $|\sigma|$, there exist constants $0 < \underline{c} \leq\overline{c}$ such that
\[
\underline{c}\, (\gamma\, N)^{|\sigma|}\leq\mathcal{N} (|\sigma|)
\leq\overline{c}\, (\gamma\, N)^{|\sigma|}.
\]
Moreover, if $w_{\sigma_{1}\cdots\sigma_{r}} (\Gamma(f))$ denotes the image of
$\Gamma(f)$ under the maps $w_{\sigma_{1}\cdots\sigma_{r}} := w_{\sigma_{1}%
}\circ\cdots\circ w_{\sigma_{r}}$ over the subinterval $l_{\sigma_{1}%
\cdots\sigma_{r}} (I)$, then there also exist constants $0 < \underline{c}^{*}
\leq\overline{c}^{*}$ such that
\begin{equation}
\label{estimate}\underline{c}^{*}\, \gamma_{\sigma_{1}}\cdots\gamma
_{\sigma_{r}}\, N^{|\sigma|}\leq\mathcal{N}_{\sigma_{1}\cdots\sigma_{r}}
(|\sigma|) \leq\overline{c}^{*}\, \gamma_{\sigma_{1}}\cdots\gamma
_{\sigma_{r}}\, N^{|\sigma|},
\end{equation}
where $\mathcal{N}_{\sigma_{1}\cdots\sigma_{r}} (|\sigma|)$ denotes the
minimum number of $N^{-|\sigma|}\times N^{-|\sigma|}$-squares from a cover of
the form (\ref{cover}) needed to cover $w_{\sigma_{1}\cdots\sigma_{r}}
(\Gamma(f))$ and $\gamma_{n} := \frac{s_{n-1}+s_{n}}{2}$, $n = 1, \ldots, N$.

Estimates of this type are important for box dimension calculations in the context of $V$-variable fractals and superfractals. We refer the interested reader to \cite{scealy} where such computations were made for affine fractal interpolants.
\end{remark}

\begin{remark}
Bilinear interpolants may be used to model or describe planar data sets that exhibit highly irregular behavior for which classical interpolation and approximation schemes such as polynomials and splines do not succeed. As in the case of affine interpolants, the determination of the free parameters, namely the scaling factors $s_0$, $s_1, \ldots s_N$, is essential for an accurate approximation of data sets using the error estimate \eqref{error}, or for modeling data with a pre-described or numerically computed box dimension. However, the particular nature of the problem dictates what type of optimization needs to be employed. For instance, an $L^2$-optimization may be applied to a functional setting as in \cite{igudesman}, or bounding volumes may be used for parameter identification as in \cite{manousopoulos}. These and related questions will be investigated elsewhere.
\end{remark}

\begin{center}
\textbf{Acknowledgements}
\end{center}

We would like to thank K. Le\'sniak for pointing out Theorem 3.82 in \cite{AB}. The second author wishes to thank The Australian National University for the kind hospitality during two visits to the Mathematical Sciences Institute in
February/March 2008 and July/August 2012.

\end{document}